\documentclass[reqno,centertags,11pt]{amsart}
\usepackage[margin=1in]{geometry}
\usepackage{amsfonts, amsmath, amssymb, amsthm, mathtools, cite}
\usepackage{pgfplots}
\pgfplotsset{compat=1.15}

\newcommand{\R}{\mathbb{R}}
\newcommand{\T}{\mathbb{T}}

\DeclarePairedDelimiter{\paren}{(}{)}
\DeclarePairedDelimiter{\abs}{\lvert}{\rvert}
\DeclareMathOperator{\sgn}{sgn}
\newcommand{\eps}{\varepsilon}

\let\div\relax
\DeclareMathOperator{\div}{div}
\newcommand{\grad}{\nabla}
\newcommand{\lap}{\Delta}

\DeclarePairedDelimiter{\nm}{\|}{\|}
\DeclarePairedDelimiter{\gen}{\langle}{\rangle}
\newcommand{\dH}{\dot{H}}

\usepackage{amsrefs, varioref, hyperref, cleveref}
\theoremstyle{definition}
\newtheorem{theorem}{Theorem}[section]
\newtheorem{lemma}[theorem]{Lemma}
\newtheorem{proposition}[theorem]{Proposition}
\newtheorem{remark}[theorem]{Remark}
\numberwithin{equation}{section}
\newcommand\numberthis{\addtocounter{equation}{1}\tag{\theequation}}

\title[Finite time blow-up in a 1D model of the IPM equation]{Finite time blow-up in a 1D model of the incompressible porous media equation}
\author{Alexander Kiselev}
\address[A. Kiselev]{Department of Mathematics, Duke University, Durham NC 27708, USA}
\email{alexander.kiselev@duke.edu}
\author{Naji A. Sarsam}
\address[N. A. Sarsam]{Department of Mathematics, Duke University, Durham NC 27708, USA}
\email[author to whom any correspondence should be addressed]{naji.sarsam@duke.edu}

\begin{document}
\begin{abstract}
We derive a PDE that models the behavior of a boundary layer solution to the incompressible porous media (IPM) equation posed on the 2D periodic half-plane. This 1D IPM model is a transport equation with a non-local velocity similar to the well-known C\'{o}rdoba-C\'{o}rdoba-Fontelos (CCF) equation. We discuss how this modification of the CCF equation can be regarded as a reasonable model for solutions to the IPM equation. Working in the class of bounded smooth periodic data, we then show local well-posedness for the 1D IPM model as well as finite time blow-up for a class of initial data.
\end{abstract}
\subjclass[2020]{76B03, 35Q35, 76S05}
\keywords{one-dimensional \rm{fl}uid model, singularity formation, IPM equation}
\maketitle

\section{Introduction}\label{sec_intro}

\subsection{Background and main results}\label{sec_intro1}

The 2D incompressible porous media (IPM) equation is an active scalar equation that models the transport of a scalar density $\rho(x,t)$ by an incompressible fluid velocity field $u(x,t)$ under the effects of Darcy's law and gravity. In the periodic half-plane $\Omega := \T \times (0, \infty)$ with $\T := [-\pi, \pi)$ denoting the circle, the initial value problem with a no-flux boundary condition is given by
\begin{equation}\label{eq_intro_IPM}
\begin{cases}
\partial_t \rho + (u \cdot \grad) \rho = 0,\\
u = -\grad p -(0, g \rho), \quad \div u = 0, \quad u_2|_{\partial \Omega} = 0, \\
\rho(\cdot, 0) = \rho_0,
\end{cases}
\end{equation}
where $p(x,t)$ is the scalar pressure, $g > 0$ is the constant of gravitational acceleration, and $\rho_0(x)$ is the initial density. We note that the functions $\rho$, $u$, and $p$ being defined on $\Omega$ means that they are $2\pi$-periodic in the $x_1$ variable and hence defined on the half plane $\R \times (0, \infty)$.

The IPM equation and its variants have garnered much interest in the literature. There have been results regarding the local well-posedness of IPM in various H\"{o}lder spaces \cite{CGO7} and Sobolev spaces \cite{C17} as well as a recent strong ill-posedness result in the critical space $H^2$ for small perturbations of a spectrally stable steady state \cite{BCM24}. There has also been much progress towards the global existence and asymptotic stability of solutions near stratified states \cites{BCP24, JK24, E17, CCL18}. Nonetheless, the question of finite time blow-up for IPM in the class of smooth solutions has remained open, with two notable results towards this goal including the demonstration of infinite-in-time growth of derivatives provided global-in-time existence \cite{KY23} and, very recently, the existence of a smooth compactly supported initial data and a compactly supported force in $L^\infty_t C^\infty_x$ that lead to finite-time blow-up in the forced IPM equation \cite{CM24}. Furthermore, the Muskat problem (a variation of the IPM equation that models the evolution of the interface between two incompressible, immiscible fluids of different densities) on the half-plane has been shown to exhibit finite time singularity formation arising from arbitrarily small initial data \cites{Z24a,Z24b}. This work has led the authors of this note to suspect that (\ref{eq_intro_IPM}) may allow for the finite time blow-up of smooth data in the periodic half-plane $\Omega$ due to the presence of a boundary, where the blow-up data looks negligible far in the interior while forming a boundary layer with the boundary trace looking like the profile in \Cref{fig_intro_InitialData}.

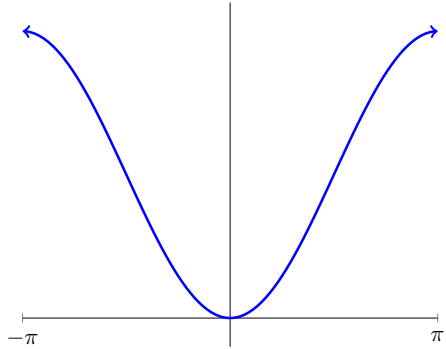
\begin{figure}[!ht]
\centering
\begin{tikzpicture}[yscale=.8,xscale=.8]
    \begin{axis}[
		xmin=-pi,
        xmax=pi,
		xtick={-pi, 0, pi},
        xticklabels={$-\pi$, , $\pi$},
        axis x line=middle,
        ymin=-0.2,
        ymax=2.2,
        ymajorticks=false,
        axis y line=middle,
        axis line style= -,
        ]
        \addplot[no marks,very thick,blue,<->] expression[domain=-pi:pi,samples=100]{1-cos(deg(x))};
    \end{axis}
\end{tikzpicture}
\caption{Example boundary trace of initial data suspected to lead to finite time blow-up in (\ref{eq_intro_IPM}). The boundary trace is even, attains 0 at the origin, and is increasing from $x_1 = 0$ to $x_1 = \pi$ (cf. \Cref{thm_intro_FTBU}).}
\label{fig_intro_InitialData}
\end{figure}

This suspected smooth blow-up data for the IPM equation parallels the proposed Hou-Luo scenario for the 3D axisymmetric Euler equation and 2D Boussinesq equation on domains with boundary \cites{HL14a,HL14b}. One avenue through which researchers gained an initial understanding of the Hou-Luo scenario was the derivation and study of 1D models for the evolution of the trace of solutions to 3D axisymmetric Euler in \cites{CKY14,CHKLSY15,DKX16}. These three 1D models were all shown to exhibit finite time blow-up for respectively appropriate smooth initial data. Since then, computer assisted proofs have been proposed which establish the Hou-Luo blow-up scenario for smooth finite-energy solutions to the true 3D axisymmetric Euler equation in a cylindrical domain \cites{CH22,CH23}.

Altogether, the study of 1D models as analogies for fluid equations has a rich history. Arguably, the first such model was the Burgers equation \cite{KNS08}, with later popular models including the Constantin-Lax-Majda (CLM) equation \cite{CLM85}, the De Gregorio equation \cite{D90}, the C\'{o}rdoba-C\'{o}rdoba-Fontelos (CCF) equation \cite{CCF07}, and the Okamoto-Sakajo-Wunsch equation (sometimes also called the OSW equation or the generalized De Gregorio/CLM equation) \cite{OSW08}. These equations focused on how best to model the competing effects of transport and vortex stretching. The later models addressing the Hou-Luo scenario, as summarized in \cite{CHKLSY15}, instead focus on how to derive a 1D approximation for the 2D Biot-Savart law applied to boundary layer solutions.

Somewhat surprisingly, a 1D model analogy for the IPM equation is currently absent from the literature, to the authors' knowledge. So, in this note, we propose and study the following active scalar transport equation with non-local velocity posed on the circle $\T$,
\begin{equation}\label{eq_intro_1DIPM}
\partial_t \rho + u \partial_x \rho = 0, \quad u = g H_a \rho, \quad \rho(\cdot, 0) = \rho_0,
\end{equation}
as a model for the evolution of the boundary trace of a solution to (\ref{eq_intro_IPM}) that is negligible in the interior of $\Omega$. In (\ref{eq_intro_1DIPM}), $g > 0$ is the constant of gravitational acceleration, $\rho(x,t)$ is the scalar boundary density, $\rho_0(x)$ is the initial boundary density, and $u(x,t)$ is the scalar 1D fluid velocity whose Biot-Savart law is given by a singular integral operator $H_a$ defined as follows:
\begin{equation}\label{eq_intro_KaKernel}
H_af(x) := P.V. \int_\R f(x-y)K_a(y)\, dy, \quad K_a(y) := \frac{1}{\pi y} \frac{a^2}{y^2+a^2},
\end{equation}
where $a > 0$ is a fixed constant. Note that any functions defined on $\T$ are considered to be $2\pi$-periodic functions over the real line, which allows for the use of $H_a \rho$ in (\ref{eq_intro_1DIPM}). We construct this model in \Cref{sec_const} by following the same scheme used to derive the models discussed in \cites{HL14a,HL14b,CHKLSY15,DKX16}. In \Cref{sec_const}, we also remark on the properties of the operator $H_a$ and how (\ref{eq_intro_1DIPM}) parallels the well-known CCF equation, whose well-posedness and singularity formation has been studied in \cites{BLM95,M98,LR08,CCF07,D08,SV16,HR16} among other papers. Our discussion suggests that CCF may indeed be a reasonable 1D model for boundary evolution under IPM.

In \Cref{sec_LWP}, we establish our first result.
\begin{theorem}\label{thm_intro_LWP}
For any fixed choice of $a, g > 0$, the 1D model (\ref{eq_intro_1DIPM}) is locally well-posed in time with respect to the class $C^\infty(\T)$ of smooth $2\pi$-periodic functions on the real line.
\end{theorem}

Our main result shows that the family of initial data to (\ref{eq_intro_1DIPM}) depicted in \Cref{fig_intro_InitialData} undergo finite time blow-up in the class $C^\infty(\T)$, which we prove in \Cref{sec_FTBU}.
\begin{theorem}\label{thm_intro_FTBU}
Suppose $\rho_0$ is non-identically zero as well as smooth, bounded, nonnegative, and even on $\T$. Moreover, suppose $\rho_0(0) = 0$ and that $\rho_0' \ge 0$ on $[0, \pi)$. Then, letting $\rho$ denote the corresponding unique local-in-time solution to (\ref{eq_intro_1DIPM}) on $\T$ for any fixed choice of $a, g > 0$, there exists a finite time $0 < T_* < \infty$ such that $\rho$ is nonnegative, bounded, and smooth for all $0 \le t < T_*$ while
\[
\lim_{t \to T_*} \int_0^t \nm{\partial_x \rho(\cdot,s)}_\infty\, ds = \infty.
\]
In particular, $\rho(x,t)$ cannot remain in $C^\infty(\T)$ for all time.
\end{theorem}

\begin{remark}
We study (\ref{eq_intro_1DIPM}) on the circle $\T$ as it captures all interesting aspects of the problem while minimizing technical local well-posedness issues. This choice also corresponds to the natural choice of domain $\Omega$ for (\ref{eq_intro_IPM}). Nonetheless, one can adapt the proofs in this note to similarly establish local well-posedness on the real line and show that finite time blow-up can still occur for a family of non-negative even smooth solutions on the real line. See \Cref{rem_FTBU_FiniteEnergy} for more details about the blow-up solutions on the real line.
\end{remark}

\subsection{Notation}

We work with the spatial domain $\T := [-\pi, \pi)$ representing the circle. Any functions defined on $\T$ are considered to be $2\pi$-periodic functions over the whole real line. In particular, we write $C^\infty(\T)$ to denote smooth $2\pi$-periodic functions which must be bounded and whose derivatives must also be bounded. When applying any (singular) integral operator to a function $f$ on $\T$, we mean to apply the operator to the periodization of $f$ to the whole real line.

We write $L^p(\T)$ with norm $\nm{f}_p$ to denote the standard Lebesgue space of $2\pi$-periodic functions, for any $1 \le p \le \infty$. The $L^2(\T)$ inner product is written as $\gen{f, g}$. For any bounded linear operator $S : L^p(\T) \to L^q(\T)$, we denote the operator norm of $S$ by $\nm{S}_{L^p \to L^q}$. For any integer $s > 0$, we write $H^s(\T)$ to mean the standard Sobolev space of $2\pi$-periodic functions that have $s$ many weak derivatives all lying in $L^2(\T)$ and we use the corresponding norm
\[
\nm{f}_{H^s} := \nm{f}_2 + \nm{f}_{\dH^s}, \qquad \nm{f}_{\dH^s} := \nm{(-\lap)^{s/2}f }_2 = \sqrt{\gen{f,\ (-\lap)^s f}}.
\]
All norms of a function $f(x,t)$ depending on space and time will refer to spatial norms.

Lastly, we write $A \lesssim B$ to mean there exists a universal constant $C > 0$ such that $A \le CB$. If the constant $C$ depends on some parameter(s) $s$, we instead write $A \lesssim_{s} B$.

\section{Constructing the model and its similarity to the CCF equation}
\label{sec_const}

\subsection{Model construction}\label{sec_const1}

Following the same procedure as outlined in \cites{HL14a,HL14b,CHKLSY15,DKX16}, we use (\ref{eq_intro_IPM}) to derive the 1D model (\ref{eq_intro_1DIPM}). Indeed, (\ref{eq_intro_IPM}) gives the following transport equation along the boundary $\partial \Omega = \T \times \{0\}$
\begin{equation}\label{eq_const_1DIPMder}
\partial_t \rho + u_1 \partial_1 \rho = 0.
\end{equation}
In turn, we wish to find an approximate formula of $u_1$ on the boundary that depends only on the value of $\rho$ on the boundary. To do so, recall that the Biot-Savart law for (\ref{eq_intro_IPM}) is given by
\[
u = g \grad^\perp (-\lap_D)^{-1} \partial_1 \rho,
\]
where $\grad^\perp := (-\partial_2, \partial_1)$ and $(-\lap_D)^{-1}$ is the inverse of the Dirichlet Laplacian on the periodic half-plane $\Omega$, with the corresponding Green's function given by
\[
G(x,y) := -\frac{1}{2\pi}\paren*{\log |y-x| - \log |y-\bar{x}|}, \quad \bar{x} := (x_1, -x_2).
\]
Note that this Biot-Savart law can be quickly obtained from (\ref{eq_intro_IPM}) by taking the scalar curl of both sides of the equation relating $u$ to $p$ and $\rho$. Now, making the boundary layer assumption $\rho(x, t) = \rho(x_1, 0, t) \chi_{[0, a]}(x_2)$ where $a > 0$ is some fixed constant, the Biot-Savart law gives
\begin{align*}
u_1(x, t) &= -g \partial_{x_2} \int_\R \int_0^\infty G(x, y) \partial_{y_1} \rho(y, t)\, dy_2\, dy_1 \\
&= -g \partial_{x_2} \int_\R \partial_{y_1} \rho(y_1, 0, t) \int_0^a G(x, y) dy_2\, dy_1 \\
&= \frac{g}{2\pi} P.V. \int_\R \partial_{y_1} \rho(y_1, 0, t) \log\paren*{ \sqrt{\frac{(y_1-x_1)^2+x_2^2}{(y_1-x_1)^2+(a-x_2)^2}}\sqrt{\frac{(y_1-x_1)^2+x_2^2}{(y_1-x_1)^2+(a+x_2)^2}} }\, dy_1.
\end{align*}
It follows that
\begin{align*}\label{eq_const_u1}
u_1(x_1, 0, t) &= \frac{g}{\pi} P.V. \int_\R \partial_{y_1} \rho(y_1, 0, t) \log\paren*{ \frac{|y_1-x_1|}{\sqrt{(y_1-x_1)^2+a^2}}}\, dy_1 \\
&= g P.V. \int_\R \rho(y_1, 0, t) \frac{1}{\pi(x_1-y_1)} \frac{a^2}{(x_1-y_1)^2+a^2}\, dy_1. \numberthis
\end{align*}
Collecting (\ref{eq_const_1DIPMder}) and (\ref{eq_const_u1}) motivates the model (\ref{eq_intro_1DIPM}) and the definition of the operator $H_a$ in (\ref{eq_intro_KaKernel}).

\begin{remark}\label{rem_const_achoice}
The motivation for deriving (\ref{eq_intro_1DIPM}), as described in \Cref{sec_intro1}, is to model the evolution of the boundary trace of a solution to (\ref{eq_intro_IPM}). Since blow-up for (\ref{eq_intro_IPM}) on the domain $\Omega$ is suspected to occur similarly as in the Hou-Luo scenario, one should (at least initially) consider (\ref{eq_intro_1DIPM}) to be a ``good'' model only for values of $a > 0$ much smaller than $1$. Of course, we note that the boundary-layer assumption $\rho(x, t) = \rho(x_1, 0, t) \chi_{[0, a]}(x_2)$ taken above will not exactly solve IPM for any choice of $a > 0$. However, in principle, any choice of $a > 0$ works for the proofs of local well-possedness (\Cref{thm_intro_LWP}) in \Cref{sec_LWP} and finite time blow-up (\Cref{thm_intro_FTBU}) in \Cref{sec_FTBU}. See \Cref{sec_const3} for how this leads to a new interpretation of the CCF equation.
\end{remark}

\begin{remark}
We emphasize that the above derivation of (\ref{eq_intro_1DIPM}) relies on the choice of domain $\Omega := \T \times (0, \infty)$ for the 2D IPM equation (\ref{eq_intro_IPM}). This choice dictates the Green's function $G(x,y)$ that eventually leads to the kernel $K_a$ defining the operator $H_a$ in (\ref{eq_intro_KaKernel}). One may instead consider a different domain for (\ref{eq_intro_IPM}), such as the interesting choice of the finite periodic channel $\T \times (0, 1)$. The Green's function, let us denote $\Gamma(x,y)$, for the inverse of the Dirichlet Laplacian on $\T \times (0, 1)$ is more complicated than that of $\Omega$ due to the extra boundary. Yet, $\Gamma(x,y)$ approximates $G(x,y)$ when $x_2, y_2 > 0$ are much smaller than $1$. Thus, we suspect that one may derive a similar 1D model with respect to the domain $\T \times (0, 1)$ when choosing $a > 0$ sufficiently small. This may lead to local well-posedness and finite time singularity formation results similar to those of \Cref{thm_intro_LWP} and \Cref{thm_intro_FTBU}, but now with a possible requirement on the smallness of $a$.
\end{remark}

\subsection{Properties of \texorpdfstring{$H_a$}{Ha}} \label{sec_const2}

The Biot-Savart law for (\ref{eq_intro_1DIPM}) is given by the convolution-type singular integral operator $H_a$ with kernel $K_a$ as defined in (\ref{eq_intro_KaKernel}). One can think of $H_a$ as an operator that ``interpolates'' between the trivial zero operator and the prototypical 1D singular integral operator: the Hilbert transform $H$ with kernel $1/(\pi y)$. This can be seen in three ways. First, the kernel $K_a$ converges pointwise to the Hilbert kernel as $a \to \infty$ while it instead converges pointwise to zero when taking $a \to 0$.

Second, $H$ and $H_a$ of course satisfy the assumptions of Calder\'{o}n-Zygmund theory, with both being bounded linear operators on $L^p$ for any $1 < p < \infty$. Furthermore, a short computation gives that $H - H_a$ is a smoothing (and not a principal value) convolution operator with the kernel $y/(\pi(y^2+a^2))$ whenever $H-H_a$ is applied to any $f \in L^p$ with $1 < p < \infty$. Young's convolution inequality then gives that $H-H_a$ is a bounded linear operator from $L^p$ to $C^\infty \cap L^q$ for any $1 < p < q < \infty$ with
\[
\nm{H - H_a}_{L^p \to L^q} \le \nm*{\frac{y}{\pi(y^2+a^2)}}_{1+1/p-1/q}.
\]
As $C^\infty \cap L^q$ is compactly embedded into $L^q$, we have that $H_a$ is equal to a compact perturbation of $H$ in the sense of the $L^p$ to $L^q$ operator norm, with this perturbation going to zero in norm as $a \to \infty$. We also point out that $H-H_a$ being a smoothing operator is particularly useful since it allows one to quickly observe that $[H-H_a] \circ \partial_x$ is smoothing too. Using this fact along with the 1D Morrey-Sobolev embedding inequality, one can show
\begin{equation}\label{eq_const_HsToLinfty}
\nm{H_a[\partial_x f]}_\infty \le \nm{(-\lap)^{1/2}f}_\infty + \nm{(H-H_a)[\partial_x f]}_\infty \lesssim_{a,s} \nm{f}_{H^s}
\end{equation}
for any $f \in H^s(\T) \subseteq L^\infty(\T)$ and any $s \ge 2$. Note that although $H_a$ is a singular integral operator that is not necessarily bounded on $L^\infty$, the extra $H^s$ regularity allows for (\ref{eq_const_HsToLinfty}) to hold true.

Lastly, being a convolution operator, $H_a$ is also a Fourier multiplier operator on the real line with the symbol $\widehat{K}_a(\xi) = -i \sgn(\xi)\paren{1-e^{-2\pi a|\xi|}}$. This symbol converges pointwise to the symbol $-i\sgn(\xi)$ for the Hilbert transform as $a \to \infty$, while it instead converges pointwise to zero as $a \to 0$. Therefore, by the dominated convergence theorem, it holds that
\[
\nm{(H-H_a)f}_2 = \nm{e^{-2\pi a|\xi|}\widehat{f}}_2 \to 0 \quad \text{ as } a \to \infty
\]
and
\[
\nm{H_af}_2 = \nm{(1-e^{-2\pi a|\xi|})\widehat{f}}_2 \to 0 \quad \text{ as } a \to 0.
\]
In words, we have that $H_a$ converges to $H$ as $a \to \infty$ while instead converging to zero as $a \to 0$, both with respect to the $L^2$ strong operator topology.

Altogether, $H_a$ can be viewed as an interpolation between the zero operator and Hilbert transform. In \Cref{sec_LWP}, we use the $L^p$ boundedness of $H_a$ for $1 < p < \infty$ and also inequality (\ref{eq_const_HsToLinfty}).

\subsection{Similarity to CCF} \label{sec_const3}

The well-known Birkhoff-Rott equations model how vortex sheets evolve under surface tension. In \cites{BLM95}, the following 1D equation posed on the real line was proposed as a model for Birkhoff-Rott:
\begin{equation}\label{eq_const_CCF}
\partial_t \theta - u \partial_x \theta + \nu(-\lap)^{\alpha/2}\theta = 0, \quad u = H\theta, \quad \theta(\cdot, 0) = \theta_0,
\end{equation}
where $\theta$ is a 1D analogue of the effective vortex sheet strength, $\nu \ge 0$ is a constant representing the strength of viscosity, and $0 \le \alpha \le 2$ is a parameter affecting how (hypo-)dissipative the model is. In \cite{CCF07}, equation (\ref{eq_const_CCF}) was shown via complex-analytic methods to allow finite time singularity formation when $\nu = 0$ for nonnegative even compactly supported sufficiently regular data attaining their maximimum at the origin. The equation (\ref{eq_const_CCF}) became known as the C\'{o}rdoba-C\'{o}rdoba-Fontelos (CCF) model, and was further studied by \cites{BLM95,M98,LR08,CCF07,D08,SV16,HR16} among other works. In particular, \cite{LR08} generalized the blow-up argument of \cite{CCF07} to the case of $\nu > 0$ and $0 < \alpha < 1/2$ while \cite{K18} proved the same result via a maximal principle approach. These two works used the same family of initial data for blow-up as \cite{CCF07}, with \cite{LR08} requiring an additional technical assumption and \cite{K18} removing the need for compact support.

Now, equation (\ref{eq_const_CCF}) with $\nu = 0$ preserves in time the spatial maximum of the initial blow-up data $\theta_0$, which is attained at $x = 0$ \cites{CCF07,K18}. Analogously, we show in \Cref{lem_LWP_AprioriLinfty} below that (\ref{eq_intro_1DIPM}) also preserves in time the maximum of the initial blow-up data $\rho_0$ which is attained at $x = \pm \pi$, under the assumptions of \Cref{thm_intro_FTBU}. These observations imply that $\theta_0(0) - \theta(x, t)$ is a function of $x$ for each fixed time $t$ that falls into our class of blow-up data for (\ref{eq_intro_1DIPM}). Correspondingly, $\rho_0(\pi) - \rho(x,t)$ gives us a blow-up data that solves (\ref{eq_const_CCF}) with $\nu = 0$ and with the Hilbert transform $H$ replaced by the operator $gH_a$. Recalling that the well-posedness of the model (\ref{eq_intro_1DIPM}) does not depend on the choice of $a > 0$ (see \Cref{rem_const_achoice}), we are encouraged by the discussion in \Cref{sec_const2} to formally take $a \to \infty$. Doing so, we have formally gone from a blow-up solution $\rho$ of (\ref{eq_intro_1DIPM}) and obtained a corresponding blow-up solution of CCF. As a result, one may argue that CCF with $\nu = 0$ actually serves as a reasonable 1D model for the trace of a boundary layer solution to the IPM equation (\ref{eq_intro_IPM}) (perhaps upon changing the constant in front of $u$ in (\ref{eq_const_CCF}) from $-1$ to $g$). Further supporting this argument, we note that IPM preserves the symmetry of an initial data that is even in $x_1$ and odd in $x_2$ while CCF with $\nu = 0$ (see \cite{K18}) and equation (\ref{eq_intro_1DIPM}) (see \Cref{lem_FTBU_monotonicty}) both also preserve the symmetry of initial data that is even on $\T \times \{0\} = \partial \Omega$.

Another well-known active scalar equation is the 2D surface quasi-geostrophic (SQG) equation \cite{CMT94}, which has a Biot-Savart law given by $(R_2, -R_1)$ with $R_j := \partial_j (-\lap)^{-1/2}$ for $j = 1,2$ denoting the standard Riesz transforms. The Riesz transforms are zeroth order singular integral operators and so too is the Hilbert transform $H$. So, one may suggest that (\ref{eq_const_CCF}) with $\nu = 0$ serves as a good 1D model for SQG. However, we note that SQG preserves the symmetry of initial data that is odd in $x_1$ and odd in $x_2$, with this behavior being the opposite of the symmetry preservation in CCF (\ref{eq_const_CCF}), as discussed above. Yet, taking the spatial derivative of (\ref{eq_const_CCF}) when $\nu = 0$ leads to the Okamoto-Sakajo-Wunsch equation (OSW) \cite{OSW08} with transport-strength parameter set equal to $-1$. This OSW equation does preserve the odd symmetry of initial data just like SQG. In turn, it may be better to instead interpret OSW (with various choices of transport-strength parameter) as a good 1D model for certain classes of solutions to SQG. Indeed, the work of \cite{CC10} shows that SQG solutions satisfying the ``stagnation point similitude ansatz'' correspond to solutions of the De Gregorio equation, which is the OSW model with parameter equal to $1$. Similarly, local-in-time solutions to SQG satisfying the ``radial homogeneity ansatz'' were shown in \cite{EJ10} to correspond to solutions of a slight variation of the OSW equation with parameter set to $2$. The work \cite{EJ17} thoroughly discusses connections between different parameter choices for the OSW equation and establishes finite time singularity formation in solutions with respect to a range of such choices.

Altogether, we emphasize that CCF may be interpreted as a reasonable model for the boundary evolution of solutions to IPM on the periodic half-space $\Omega$.

\section{Local well-posedness}\label{sec_LWP}

Here, we present the lemmas required to prove \Cref{thm_intro_LWP} on local well-posedness. We first prove the uniqueness of solutions to (\ref{eq_intro_1DIPM}).

\begin{lemma}\label{lem_LWP_AprioriUnique}
Fix $a, g > 0$ and suppose $\rho_1, \rho_2$ are solutions in $C^\infty(\T)$ to (\ref{eq_intro_1DIPM}) on a shared time interval $[0, T)$ with respect to the same initial data $\rho_0 \in C^\infty(\T)$. Then $\rho_1 = \rho_2$ on $[0, T)$.
\end{lemma}

\begin{proof}
All of the following computations and estimates hold on the time interval $[0, T)$. Defining $\rho := \rho_1 - \rho_2$, we have that
\[
\frac{1}{g}\partial_t \rho = -\partial_x \rho_1 H_a \rho_1 + \partial_x \rho_2 H_a \rho_2 = -\partial_x \rho H_a \rho_1 - \partial_x \rho_2 H_a \rho.
\]
It follows that
\[
\frac{1}{2g}\frac{d}{dt} \nm{\rho}_2^2 = - \gen{\partial_x \rho H_a \rho_1,\ \rho} - \gen{\partial_x \rho_2 H_a \rho,\  \rho}.
\]
We observe that
\[
-\gen{\partial_x \rho H_a \rho_1,\ \rho} = \frac{1}{2}\gen{H_a(\partial_x \rho_1),\ \rho^2} \le \nm{H_a[\partial_x \rho_1]}_\infty \nm{\rho}_2^2 \lesssim_{a,s} \nm{\rho_1}_{H^s}\nm{\rho}_2^2,
\]
where the final inequality holds for any integer $s \ge 2$ by (\ref{eq_const_HsToLinfty}), and
\[
|\gen{\partial_x \rho_2 H_a \rho,\  \rho}| \le \nm{H_a}_{L^2 \to L^2} \nm{\partial_x \rho_2}_\infty \nm{\rho}_2^2.
\]
Altogether, it follows
\[
\frac{1}{g}\frac{d}{dt} \nm{\rho}_2^2 \lesssim_{a,s} \paren*{\nm{H_a}_{L^2 \to L^2} \nm{\partial_x \rho_2}_\infty + \nm{\rho_1}_{H^s}} \nm{\rho}_2^2.
\]
Gr\"{o}nwall's inequality completes the proof since $\rho(x, 0) = 0$ for all $x \in \T$.
\end{proof}

We next establish a-priori estimates on the growth of the $L^\infty$ and $L^2$ norms of a solution.

\begin{lemma}\label{lem_LWP_AprioriLinfty}
Fix $a, g > 0$ and suppose for some integer $s \ge 4$ that $\rho$ is a solution in $H^s(\T)$ to (\ref{eq_intro_1DIPM}) on the time interval $[0, T)$. Then, given $\rho_0 \in H^s(\T) \subseteq L^\infty(\T)$, we have $\nm{\rho}_\infty = \nm{\rho_0}_\infty$ on $[0, T)$.
\end{lemma}

\begin{proof}
We have by the $L^2$ boundedness of $H_a$ (and as spatial differentiation commutes with $H_a$) that $H_a \rho \in H^4(\T)$. By the 1D Morrey-Sobolev embedding, it follows that $H_a \rho$ is thrice continuously differentiable in space and, due to (\ref{eq_intro_1DIPM}), twice continuously differentiable in time. Since (\ref{eq_intro_1DIPM}) is a transport equation with advection velocity $u = g H_a \rho$ having sufficient regularity, one can construct an invertible flow map $\Phi_t(x)$ for each fixed $x \in \T$ defined for all time $t \in [0, T)$. It holds then that $\rho(x, t) = \rho_0(\Phi_t^{-1}(x))$ for all $x \in \T$ and $t \in [0, T)$, thus completing the proof.
\end{proof}

\begin{lemma}\label{lem_LWP_AprioriL2}
Fix $a, g > 0$ and suppose $\rho$ is a solution in $H^s(\T)$ to (\ref{eq_intro_1DIPM}) on the time interval $[0, T)$ for some integer $s \ge 2$. It then holds that
\[
\frac{1}{g}\frac{d}{dt} \nm{\rho}_2^2 \lesssim_{a,s} \nm{\rho}_{H^s} \nm{\rho}_2^2 = \nm{\rho}_{\dH^s} \nm{\rho}_2^2 + \nm{\rho}_2^3
\]
on $[0, T)$.
\end{lemma}

\begin{proof}
We have by (\ref{eq_const_HsToLinfty}) that
\[
\frac{1}{g}\frac{d}{dt} \nm{\rho}_2^2 = -2\gen{H_a \rho \partial_x \rho,\ \rho} = \gen{H_a[\partial_x \rho],\ \rho^2} \le \nm{H_a[\partial_x \rho]}_{\infty} \nm{\rho}_2^2 \lesssim_{a,s} \nm{\rho}_{H^s} \nm{\rho}_2^2
\]
on the time interval $[0, T)$.
\end{proof}

Lastly, it remains to bound the $\dH^s$ norm of a solution. To do so, we make use of the following 1D Gagliardo-Nirenberg-Sobolev interpolation inequality, which we state for ease of reference.

\begin{lemma}\label{lem_LWP_GNS}
Fix $1 \le p < \infty$ and integers $0 < j < s$. Suppose there exists $\theta \in [j/s, 1]$ such that
\[
\frac{1}{p} = j + \theta\paren*{\frac{1}{2} - s}.
\]
Then, for any $f \in H^s(\T) \subseteq L^\infty(\T)$, we have
\[
\nm{\partial_x^j f}_p \lesssim_{p,j,s} \nm{f}_{H^s}^\theta\nm{f}_\infty^{1-\theta}.
\]
\end{lemma}

\begin{lemma}\label{lem_LWP_AprioriHs}
Fix $a, g > 0$ and suppose $\rho$ is a solution in $H^s(\T) \subseteq L^\infty(\T)$ to (\ref{eq_intro_1DIPM}) on the time interval $[0, T)$ for some integer $s \ge 2$. It then holds that
\begin{align*}
\frac{1}{g}\frac{d}{dt} \nm{\rho}_{\dH^s}^2 &\lesssim_{a,s} \nm{\rho}_{H^s}\nm{\rho}_{\dH^s}^2 + \nm{\rho}_{\dH^s} \nm{\rho}_{H^s}^{1 + \frac{2}{2s-1}}\nm{\rho}_{\infty}^{1-\frac{2}{2s-1}} \\
&\le \paren*{\nm{\rho}_{\dH^s} + \nm{\rho}_\infty}^3 + \nm{\rho}_{\dH^s}^{2 + \frac{2}{2s-1}}\nm{\rho}_\infty^{1-\frac{2}{2s-1}}.
\end{align*}
on $[0, T)$.
\end{lemma}

\begin{proof}
All of the following computations and estimates hold on the time interval $[0, T)$. We observe
\begin{equation}\label{eq_LWP_HsBound}
\frac{1}{2g}\frac{d}{dt} \nm{\rho}_{\dH^s}^2 = -\gen{H_a \rho \partial_x \rho,\ (-\lap)^s \rho} \le \abs{\gen{\partial_x^s[H_a \rho \partial_x \rho],\ \partial_x^s \rho}} \lesssim_s \sum_{k=0}^s \abs{\gen{\partial_x^k[H_a \rho] \partial_x^{s-k+1} \rho,\ \partial_x^s \rho}}.
\end{equation}
We now estimate each integral term in the sum on the right-hand side of (\ref{eq_LWP_HsBound}). There are four cases that we consider: $k = 0$, $k =1$, $2 \le k < s$, and $k = s$.

Three of the cases can be estimated easily. When $k = 0$, we have
\[
\abs{\gen{H_a\rho \partial_x^{s+1} \rho, \partial_x^s \rho}} = \frac{1}{2}\abs{\gen{H_a[\partial_x \rho],\ (\partial_x^s \rho)^2}} \le \nm{H_a[ \partial_x \rho]}_\infty \nm{\rho}_{\dH^s}^2 \lesssim_{a,s} \nm{\rho}_{H^s}\nm{\rho}_{\dH^s}^2,
\]
where the final inequality holds by (\ref{eq_const_HsToLinfty}). The case of $k = 1$ is completely analogous as
\[
\abs{\gen{H_a[\partial_x \rho] \partial_x^s \rho, \partial_x^s \rho}} \le \nm{H_a[ \partial_x \rho]}_\infty \nm{\rho}_{\dH^s}^2 \lesssim_{a,s} \nm{\rho}_{H^s}\nm{\rho}_{\dH^s}^2.
\]
The remaining simple case is that of $k = s$, where we have
\[
\abs{\gen{\partial_x^s[H_a \rho] \partial_x \rho,\ \partial_x^s \rho}} \le \nm{H_a}_{L^2 \to L^2} \nm{\partial_x \rho}_\infty \nm{\rho}_{\dH^s}^2
\lesssim_{a,s} \nm{\partial_x \rho}_{H^{s-1}} \nm{\rho}_{\dH^s}^2 \le \nm{\rho}_{H^s} \nm{\rho}_{\dH^s}^2,
\]
with the second inequality holding true by the 1D Morrey-Sobolev embedding. The final case is when $s \ge 3$ and $2 \le k \le s-1$. Defining the exponents
\[
p_k := \frac{2s}{s-k+1}, \quad \theta_k := \frac{s-k+1}{s}, \quad \widetilde{p}_k := \frac{2s}{k-1}, \quad \widetilde{\theta}_k = \frac{k}{s}+ \frac{1}{2s^2-s},
\]
we may invoke H\"{o}lder's inequality and \Cref{lem_LWP_GNS} to obtain
\begin{align*}
\abs{\gen{\partial_x^k[H_a \rho] \partial_x^{s-k+1} \rho,\ \partial_x^s \rho}} &\le \nm{H_a}_{L^{\widetilde{p}_k} \to L^{\widetilde{p}_k}} \nm{\partial_x^s\rho}_2 \nm{\partial_x^{s-k+1} \rho}_{p_k} \nm{\partial_x^k \rho}_{\widetilde{p}_k} \\
&\lesssim_{a,s} \nm{\rho}_{\dH^s} \nm{\rho}_{H^s}^{\theta_k}\nm{\rho}_{\infty}^{1-\theta_k}\nm{\rho}_{H^s}^{\widetilde{\theta}_k}\nm{\rho}_{\infty}^{1-\widetilde{\theta}_k} \\
&= \nm{\rho}_{\dH^s} \nm{\rho}_{H^s}^{1 + \frac{2}{2s-1}}\nm{\rho}_{\infty}^{1-\frac{2}{2s-1}}.
\end{align*}
Collecting all of the casework with estimate (\ref{eq_LWP_HsBound}) gives
\begin{align*}
\frac{1}{g}\frac{d}{dt} \nm{\rho}_{\dH^s}^2 &\lesssim_s \nm{\rho}_{H^s}\nm{\rho}_{\dH^s}^2 + \nm{\rho}_{\dH^s} \nm{\rho}_{H^s}^{1 + \frac{2}{2s-1}}\nm{\rho}_{\infty}^{1-\frac{2}{2s-1}} \\
&\lesssim_s \paren{\nm{\rho}_{\dH^s}+\nm{\rho}_\infty}\nm{\rho}_{\dH^s}^2 + \nm{\rho}_{\dH^s} \paren{\nm{\rho}_{\dH^s}^{1 + \frac{2}{2s-1}}+\nm{\rho}_\infty^{1 + \frac{2}{2s-1}}}\nm{\rho}_{\infty}^{1-\frac{2}{2s-1}} \\
&\le \paren*{\nm{\rho}_{\dH^s} + \nm{\rho}_\infty}^3 + \nm{\rho}_{\dH^s}^{2 + \frac{2}{2s-1}}\nm{\rho}_\infty^{1-\frac{2}{2s-1}},
\end{align*}
which completes the proof.
\end{proof}

Collecting \Cref{lem_LWP_AprioriLinfty}, \Cref{lem_LWP_AprioriL2}, and \Cref{lem_LWP_AprioriHs}, one can use the standard argument of a smoothing scheme along with sequential compactness to establish the existence of a smooth bounded solution on $\T$ to (\ref{eq_intro_1DIPM}) for a small time interval depending on the smooth bounded initial data. Then, \Cref{lem_LWP_AprioriUnique} ensures the uniqueness of the solution, and hence \Cref{thm_intro_LWP} holds true for any fixed choice of $a, g > 0$.

\begin{remark}
Consider a solution $\rho$ in $C^\infty(\T)$ to (\ref{eq_intro_1DIPM}) on the time interval $[0, T)$. By differentiating the mean
\[
\frac{1}{2\pi} \int_{-\pi}^\pi \rho(x, t)\, dx
\]
with respect to time, invoking (\ref{eq_intro_1DIPM}), and using the Fourier symbol of $H_a$ discussed in \Cref{sec_const}, one can show that the mean of $\rho$ must be increasing. This may initially seem to be an issue when establishing local well-posedness. However, the preservation of the $L^\infty$ norm given by \Cref{lem_LWP_AprioriLinfty} is enough to ensure the local-in-time existence required for \Cref{thm_intro_LWP}. We also note that growth of the mean for a smooth solution $\rho$ to (\ref{eq_intro_1DIPM}) corresponds with phenomenon of  decay of the mean for smooth solutions to the CCF equation (\ref{eq_const_CCF}) with $\nu = 0$, which can similarly be proven using the Fourier symbol of the Hilbert transform.
\end{remark}

\section{Finite time blow-up for smooth data}\label{sec_FTBU}

Now that we have established local well-posedness, we assert the following Beale-Kato-Majda type criterion for (\ref{eq_intro_1DIPM}).

\begin{proposition}\label{prop_FTBU_BKM}
Fix $a, g > 0$ and suppose $\rho$ is a solution  to (\ref{eq_intro_1DIPM}) in $C^\infty(\T \times [0, T_*))$ corresponding to an initial data $\rho_0 \in C^\infty(\T)$. If $0 < T_* < \infty$ is the first finite blow-up time such that $\rho$ cannot be continued in $C^\infty(\T)$, then we must have that
\[
\lim_{t \to T_*} \int_0^t \nm{\partial_x \rho(\cdot, s)}_\infty\, ds = \infty.
\]
\end{proposition}
As $H_a$ is a singular integral operator with kernel $K_a$ bounded above by the kernel $1/(\pi y)$ of the Hilbert transform (see (\ref{eq_intro_KaKernel})), the proof of Proposition 5.2 in \cite{D08} can be immediately adapted with very minor modifications to establish \Cref{prop_FTBU_BKM}.

This criterion reduces the proof of \Cref{thm_intro_FTBU} to showing that the class of initial data depicted in \Cref{fig_intro_InitialData} always leads to finite time blow-up in some way. Our argument to do so is inspired by \cite{K18}, which proves finite time blow up for (\ref{eq_const_CCF}) on the whole line with $\nu \ge 0$ and $0 < \alpha < 1/2$ by leveraging the preservation of monotonicity under (\ref{eq_const_CCF}) and by extending an integral inequality for the Hilbert transform that was first derived in \cite{CCF07} (and later also used in \cite{LR08}). The corresponding estimates for $H_a$ in the periodic setting pose some challenges that are resolved below.

\begin{lemma}\label{lem_FTBU_monotonicty}
Fix $a, g > 0$ and suppose $\rho$ is a solution to (\ref{eq_intro_1DIPM}) in $C^\infty(\T \times [0, T))$ corresponding to the initial data $\rho_0 \in C^\infty(\T)$. Suppose further that $\rho_0$ is even and nonnegative with $\rho_0(0) = 0$ and $\rho_0' \ge 0$ on $[0,\pi)$. Then, for all time $t \in [0, T)$, $\rho$ remains even and nonnegative with $\rho(0, t) = 0$ and $\partial_x \rho \ge 0$ on $[0,\pi)$.
\end{lemma}

\begin{proof}
Defining $\widetilde{\rho}(x,t) := \rho(-x,t)$, one shortly sees that $\widetilde{\rho}$ solves (\ref{eq_intro_1DIPM}) on $[0, T)$ with $\widetilde{\rho}(x,0) = \rho_0(x)$. By the uniqueness obtained in \Cref{thm_intro_LWP}, we conclude that $\rho = \widetilde{\rho}$ on $[0, T)$ and hence $\rho$ is even.

As in the proof of \Cref{lem_LWP_AprioriLinfty}, we note that (\ref{eq_intro_1DIPM}) is a transport equation with advection velocity $u = g H_a \rho$ smooth on $[0,T)$. Thus, one can construct an invertible flow map $\Phi_t(x)$ such that
\[
\frac{d}{dt} \Phi_t(x) = u(\Phi_t(x), t), \quad \Phi_0(x) = x
\]
for each fixed $x \in \T$. In turn, it holds that $\rho(x, t) = \rho_0(\Phi_t^{-1}(x))$ and so it follows that $\rho$ is nonnegative for all time in $[0, T)$. Moreover, $\rho$ being even implies that $u(0, t) = gH_a\rho(0, t) = 0$. Hence the origin is fixed under the flow, with $\Phi_t(0) = 0$ for all $t \in [0, T)$ and thus $\rho(0, t) = 0$.

Lastly, differentiating both sides of (\ref{eq_intro_1DIPM}) in space, we obtain
\[
\partial_t \eta + u \partial_x \eta = -g\eta H_a \eta, \quad u := gH_a \rho
\]
for $\eta := \partial_x \rho$. This is also a transport equation under the same velocity $u$, and therefore the same flow map $\Phi_t(x)$, that yields
\[
\frac{d}{dt} \eta(\Phi_t(x), t) = - g\eta(\Phi_t(x), t) H_a \eta(\Phi_t(x), t).
\]
We solve to get
\begin{equation}\label{eq_FTBU_eta}
\eta(x, t) = \eta(\Phi_t^{-1}(x), 0) \exp\paren*{-g\int_0^t H_a \eta(\Phi_s\circ \Phi_t^{-1}(x), s)\, ds }.
\end{equation}
Again recalling that $\rho(x,t)$ is even for all time $t \in [0, T)$, we in turn have $u(0, t) = gH_a\rho(0, t) = 0$ and $u(\pi, t) = gH_a\rho(\pi, t) = 0$. As $\eta(x, 0) = \rho'_0(x) \ge 0$ on $[0, \pi)$ by assumption, it follows from  (\ref{eq_FTBU_eta}) that $\eta = \partial_x \rho \ge 0$ on $[0, \pi)$ for all time in $[0, T)$.
\end{proof}

\begin{proposition}\label{prop_FTBU_HaInequality}
Let $f \in C^\infty(\T)$ be even and nonnegative with $f(0) = 0$ and $f' \ge 0$ on $[0,\pi)$. Then for every $a, \sigma > 0$ we have
\begin{equation}\label{eq_FTBU_KeyIneq}
-\int_0^{\pi/2} \frac{H_a f(x) f'(x)}{x^{\sigma}}\, dx \ge C_{a,\sigma} \int_0^{\pi/2} \frac{f(x)^2}{x^{1+\sigma}}\, dx,
\end{equation}
where $C_{a,\sigma} > 0$ is a universal constant depending only on $\sigma$ and $a$.
\end{proposition}

\begin{remark}
Note that for large $\sigma$ the right side may be infinite; the statement still holds in the sense that in this case the left side must be infinite too.
\end{remark}

The proof of this proposition requires some involved computations, with the main difficulties arising from estimating the kernel $K_a$ in (\ref{eq_intro_KaKernel}) as well as the lack of monotonicity of $f$ on the whole real line (as we only have monotonicity on each half period). Thus, let us first see how to use \Cref{prop_FTBU_BKM}, \Cref{lem_FTBU_monotonicty}, and \Cref{prop_FTBU_HaInequality} in order to establish the main result of this note.

\begin{proof}[Proof of \Cref{thm_intro_FTBU}]
Suppose the initial data $\rho_0 \in C^\infty(\T)$ is non-identically zero, nonnegative, and even with $\rho_0(0) = 0$ and $\rho_0' \ge 0$ on $[0,\pi)$. Fix $a, g > 0$ and let $\rho(x, t)$ denote the corresponding unique local-in-time smooth, bounded, non-negative solution to (\ref{eq_intro_1DIPM}). Now, assume for contradiction that $\rho$ stays smooth and bounded for all time. Then we can define the function
\[
J(t) := \int_0^{\pi/2} \frac{\rho(x,t)}{x^{1+\delta}}\, dx
\]
where $0 < \delta < 1$ is arbitrarily fixed. We compute
\begin{align*}
J'(t) &= -g\int_0^{\pi/2} \frac{H_a\rho(x,t) \partial_x \rho(x,t)}{x^{1+\delta}}\, dx \\
&\ge gC_{a,1+\delta} \int_0^{\pi/2} \frac{\rho(x,t)^2}{x^{2+\delta}}\, dx \\
&\ge gC_{a,1+\delta} \paren*{\int_0^{\pi/2} \frac{1}{x^{\delta}}\, dx}^{-1} \paren*{\int_0^{\pi/2} \frac{\rho(x,t)}{x^{1+\delta}}\, dx}^2 \\
&= \frac{g(1-\delta)C_{a,1+\delta}}{(\pi/2)^{1-\delta}} J(t)^2,
\end{align*}
where we have invoked \Cref{lem_FTBU_monotonicty}, \Cref{prop_FTBU_HaInequality}, and H\"{o}lder's inequality. As $J(0) > 0$, it follows that $J(t)$ must blow up in finite time, leading to a contradiction. \Cref{prop_FTBU_BKM} completes the proof.
\end{proof}

All that remains is to prove \Cref{prop_FTBU_HaInequality}. To do so, we first prove the following lemma which leverages the assumption of $f$ being monotonically increasing on the half period $[0, \pi)$ (as compared to the monotonicity on the real line used for the analogous integral inequalities in \cites{CCF07,LR08,K18}). Recall that any functions $f$ defined on $\T$ are considered to be $2\pi$-periodic functions over the real line.

\begin{lemma}\label{lem_FTBU_HaEstimate}
Fix $a > 0$. Given $f$ satisfying the hypotheses of \Cref{prop_FTBU_HaInequality}, it holds that
\begin{equation}\label{eq_FTBU_HaIneq}
H_a f(x) \le \int_0^{2x} f'(y) G_a(x, y)\, dy
\end{equation}
for all $0 < x \le \pi/2$, where we define
\[
G_a(x,y) := \frac{1}{2\pi}\log\paren*{
\frac{1+\frac{a^2}{x^2}}{1+\frac{a^2}{(x-y)^2}}\cdot\frac{1+\frac{a^2}{(2\pi-x)^2}}{1+\frac{a^2}{(x-y+2\pi\sgn(y-x))^2}}}.
\]
\end{lemma}

\begin{proof}
Fix $a > 0$ and $0 < x \le \pi/2$. Recall the definition of the kernel $K_a(y)= \frac{a^2}{\pi y(y^2+a^2)}$ in (\ref{eq_intro_KaKernel}) and note that $K_a$ is odd. Thus, by a change of variables we obtain
\begin{align*}\label{eq_FTBU_HaIneq1}
\paren*{\int_{-\infty}^{-\eps} + \int_\eps^\infty} f(x-y) K_a(y)\, dy &= \paren*{\int_0^{x-\eps} + \int_{x+\eps}^{2x}} f(y)K_a(x-y)\, dy \\
&\quad + \int_x^\infty (f(x-y)-f(x+y))K_a(y)\, dy \numberthis
\end{align*}
for all $\eps > 0$ sufficiently small. To simplify the right-hand side of (\ref{eq_FTBU_HaIneq1}), we note by the periodicity of $f$ that the function $y \mapsto f(x-y)-f(x+y)$ is itself $2\pi$-periodic. Furthermore, due to the evenness and monotonicity assumptions on $f$, that same function $y \mapsto f(x-y)-f(x+y)$ is also odd symmetric on each interval $y \in (2\pi k, 2\pi(k+1))$ about the value $y = 2\pi k + \pi$, with the function being nonpositive for $y \in (2\pi k, 2\pi k + \pi)$ and nonnegative for $y \in (2\pi k + \pi, 2\pi(k+1))$, where $k$ is any nonnegative integer. As the kernel $K_a(y)$ is strictly decreasing on $(0, \infty)$, it holds
\[
\int_x^{2\pi-x} (f(x-y)-f(x+y))K_a(y)\, dy < 0, \quad \int_{2\pi}^\infty (f(x-y)-f(x+y))K_a(y)\, dy < 0,
\]
from which it follows by (\ref{eq_FTBU_HaIneq1}) that
\begin{align*}\label{eq_FTBU_HaIneq2}
&\paren*{\int_{-\infty}^{-\eps} + \int_\eps^\infty} f(x-y) K_a(y)\, dy \\
&\qquad \le  \paren*{\int_0^{x-\eps} + \int_{x+\eps}^{2x}} f(y)K_a(x-y)\, dy + \int_{2\pi-x}^{2\pi} (f(x-y)-f(x+y))K_a(y)\, dy \\
&\qquad = \paren*{\int_{x-2\pi}^{2x-2\pi} + \int_0^{x-\eps} + \int_{x+\eps}^{2x} + \int_{2\pi}^{x+2\pi} } f(y)K_a(x-y)\, dy. \numberthis
\end{align*}
Integrating by parts and using that the antiderivative of $K_a(y)$ is given by
\[
Q_a(y) := \frac{1}{\pi}\log \frac{|y|}{\sqrt{y^2+a^2}},
\]
one obtains the following equalities:
\begin{align*}
&\int_{x-2\pi}^{2x-2\pi} f(y)K_a(x-y)\, dy = f(x)Q_a(2\pi)-f(2x)Q_a(2\pi-x) + \int_x^{2x} f'(y)Q_a(x-y+2\pi)\, dy, \\
&\int_0^{x-\eps} f(y)K_a(x-y)\, dy = -f(x-\eps)Q_a(\eps) + \int_0^{x-\eps} f'(y)Q_a(x-y)\, dy, \\
&\int_{x+\eps}^{2x} f(y)K_a(x-y)\, dy = f(x+\eps)Q_a(\eps)-f(2x)Q_a(x) + \int_{x+\eps}^{2x} f'(y)Q_a(x-y)\, dy, \\
&\int_{2\pi}^{x+2\pi} f(y)K_a(x-y)\, dy = -f(x)Q_a(2\pi) + \int_0^x f'(y)Q_a(x-y-2\pi)\, dy,
\end{align*}
where we have many times used that $f(2\pi) = f(0) = 0$ and that $f$ and $f'$ are $2\pi$-periodic. Substituting these equalities into (\ref{eq_FTBU_HaIneq2}) and taking the limit of $\eps$ to zero, we obtain that
\begin{align*}
H_af(x) &\le -f(2x)(Q_a(2\pi-x)+Q_a(x)) + \int_0^{2x} f'(y)(Q_a(x-y)+Q_a(x-y+2\pi\sgn(y-x)))\, dy.
\end{align*}
Using the fundamental theorem of calculus on $f(2x)$ gives (\ref{eq_FTBU_HaIneq}) with
\begin{align*}
G_a(x,y) &= Q_a(x-y)-Q_a(x)+Q_a(x-y+2\pi\sgn(y-x))-Q_a(2\pi-x) \\
&= \frac{1}{\pi}\log \abs*{
\frac{x-y}{x}\sqrt{\frac{x^2+a^2}{(x-y)^2+a^2}}
\frac{x-y+2\pi\sgn(y-x)}{2\pi-x}\sqrt{\frac{(2\pi-x)^2+a^2}{(x-y+2\pi\sgn(y-x))^2+a^2}}
} \\
&= \frac{1}{2\pi}\log\paren*{
\frac{1+\frac{a^2}{x^2}}{1+\frac{a^2}{(x-y)^2}}\cdot\frac{1+\frac{a^2}{(2\pi-x)^2}}{1+\frac{a^2}{(x-y+2\pi\sgn(y-x))^2}}},
\end{align*}
as desired.
\end{proof}

By appropriately estimating $G_a(x,y)$ in (\ref{eq_FTBU_HaIneq}), we are able to conclude the proof.

\begin{proof}[Proof of \Cref{prop_FTBU_HaInequality}]
Fix $a > 0$ and $0 < x \le \pi/2$. By hand, one can compute the derivative in $y$ of the argument of the logarithm defining $G_a(x,y)$ in (\ref{eq_FTBU_HaIneq}). Immediately, one verifies that the derivative is positive for each choice of $y \in (x,2x]$ and negative for each choice of $y \in [0, x)$. It follows that $G_a(x, \cdot)$ is strictly increasing on $(x,2x]$ and strictly decreasing on $[0, x)$. Since $G_a(x,0) = G_a(x,2x) = 0$, we have that $G_a(x,y) \le 0$ for each $y \in [0,2x]$, with a singularity of negative infinity attained at $y = x$. Therefore, we have that 
\begin{equation}\label{eq_FTBU_GaEstimate1}
G_a(x,y) \le
\begin{cases}
\max\{G_a(x,x/q), G_a(x,qx)\} & \text{if } y \in [x/q, qx] \\
0 & \text{if } y \in [0,x/q) \cup (qx, 2x],
\end{cases}
\end{equation}
for any fixed choice of $q \in (1,2)$.

Keeping $a > 0$ and $1 < q < 2$ fixed, our next claim is that
\begin{equation}\label{eq_FTBU_GaEstimate2}
\max\{G_a(x,x/q), G_a(x,qx)\} = G_a(x,qx),
\end{equation}
for all $0 < x \le \pi/2$. To establish (\ref{eq_FTBU_GaEstimate2}), one observes that
\[
G_a(x,qx)-G_a(x,x/q) = \frac{1}{2\pi}\log
\paren*{\frac{1+\frac{(qa)^2}{(q-1)^2x^2}}{1+\frac{a^2}{(q-1)^2x^2}} \cdot \frac{1+\frac{(qa)^2}{(2\pi q + (1-q)x)^2}}{1+\frac{a^2}{(2\pi+(1-q)x)^2}}}.
\]
Splitting the above expression into a sum (and difference) of logarithms, differentiating it in $x$, and performing elementary (if somewhat lengthy) estimates, one can verify by hand that the resulting derivative is negative for $x \in (0, 2\pi/(q-1))$. It follows that $G_a(x,qx)-G_a(x,x/q)$ is strictly decreasing in $x$ on the interval $(0, 2\pi/(q-1))$. Moreover, one quickly checks that
\[
G_a(x_*,qx_*)-G_a(x_*,x_*/q) = 0, \quad x_* := \frac{2\pi q}{(q+1)(q-1)} \in (\pi/2, 2\pi/(q-1)).
\]
Hence, the strict monotonicity of $G_a(x,qx)-G_a(x,x/q)$ immediately implies (\ref{eq_FTBU_GaEstimate2}).

Collecting (\ref{eq_FTBU_GaEstimate1}), (\ref{eq_FTBU_GaEstimate2}), and \Cref{lem_FTBU_HaEstimate}, we have by the monotonicity of $f$ on $[0, \pi]$ that
\[
H_af(x) \leq G_a(x,qx)(f(qx) - f(x/q)),
\]
for any $0 < x \le \pi/2$. It follows that
\[
-\int_0^{\pi/2} \frac{H_a f(x) f'(x)}{x^{\sigma}}\, dx \ge  \int_0^{\pi/2} \frac{-G_a(x,qx)(f(qx) - f(x/q)) f'(x)}{x^{\sigma}}\, dx.
\]
A final direct computation (along the same lines as those prior) shows that $-G_a(x,qx)$ is strictly decreasing in $x$ on $[0,2\pi/q]$ with $-G_a(2\pi/q, 2\pi) = 0$. Therefore, $-G_a(\pi/2,q\pi/2) > 0$, regardless of the fixed choice of $a > 0$ and $1 < q < 2$. Once more using the monotonicty of $f$ on $[0,\pi]$, we obtain
\[
-\int_0^{\pi/2} \frac{H_a f(x) f'(x)}{x^{\sigma}}\, dx \ge -G_a(\pi/2,q\pi/2)\int_0^{\pi/2} \frac{(f(qx) - f(x/q)) f'(x)}{x^{\sigma}}\, dx.
\]
As we chose an arbitrary $q \in (1,2)$, our proof is concluded by the following fact from the proof of Proposition 6.4 in \cites{K18}:
\[
\int_0^{\pi/2} \frac{(f(qx) - f(x/q)) f'(x)}{x^{\sigma}}\, dx \ge \widetilde{C}_{q,\sigma} \int_0^{\pi/2} \frac{f(x)^2}{x^{1+\sigma}}\, dx,
\]
where $\widetilde{C}_{q,\sigma} > 0$ depends only on $q$ and $\sigma$. Indeed, (\ref{eq_FTBU_KeyIneq}) holds with $C_{a,\sigma} := -\widetilde{C}_{q,\sigma} G_a(\pi/2, q\pi/2) > 0$, where we recall that we can fix any $q \in (1,2)$ independent of $f$, $a$, and $\sigma$.
\end{proof}

\begin{remark}\label{rem_FTBU_FiniteEnergy}
One can follow the proof of \Cref{prop_FTBU_HaInequality} to establish an analogous result on the real line for smooth functions $f$ that are even, nonnegative, and bounded on all of $\R$, with $f(0) = 0$ and $f' \ge 0$ on $(0,\infty)$. In this setting, one sees that the term
\[
\int_x^\infty (f(x-y)-f(x+y))K_a(y)\, dy
\]
from (\ref{eq_FTBU_HaIneq1}) is indeed negative, which leads to a simpler form of the function $G_a(x,y)$. Overall, this leads to a finite time blow-up result on the real line with smooth initial data $\rho_0$ satisfying the same assumptions stated for $f$ above. We note that such initial data $\rho_0$ do not have finite energy, meaning they are not in $L^2(\R)$ (even though they may correspond to finite-energy blow-up data for the CCF equation (\ref{eq_const_CCF}) according to our discussion in \Cref{sec_const3}). However, one can consider smoothly cutting off such an initial data $\rho_0$ at a large spatial scale, so that it becomes compactly supported. Then, as (\ref{eq_intro_1DIPM}) is a transport equation with finite speed of propogation, the smooth compactly supported modified data should still undergo singularity formation in finite time.
\end{remark}

\section*{Acknowledgements}

AK has been partially supported by the NSF-DMS grant 2306726. NS has been partially supported by the NSF-DMS grant 2038056.


\end{document}